\documentclass[11pt]{article}
\usepackage{amsmath, amscd, amssymb, latexsym, epsfig, color, amsthm}
\numberwithin{equation}{section}

\newtheorem{theorem}{Theorem}[section]

\newtheorem{lemma}[theorem]{Lemma}

\theoremstyle{definition}

\DeclareMathOperator\lk{\mathrm{lk}}
\DeclareMathOperator\rk{\mathrm{rk}}
\DeclareMathOperator\sd{\mathrm{sd}}

\newcommand{\ZZ}{{\mathbb Z}}

\newcommand{\Sm}{{\mathcal S}}

\newcommand{\Tor}{\ensuremath{\mathrm{Tor}}\hspace{1pt}}

\newcommand{\KK}{{\mathbb K}}

\title{Face numbers and the fundamental group}
\author{Satoshi Murai\thanks{Research is partially
supported by JSPS KAKENHI 	16K05102.}\\
\small Department of Pure and Applied Mathematics\\[-0.8ex]
\small Graduate School of Information Science and Technology\\[-0.8ex]
\small Osaka University, Suita, Osaka 565-0871, Japan\\[-0.8ex]
\small \texttt{s-murai@ist.osaka-u.ac.jp}
\and Isabella Novik\thanks{Research is partially\textsl{}
supported by NSF grant DMS-1361423}\\
\small Department of Mathematics\\[-0.8ex]
\small University of Washington\\[-0.8ex]
\small Seattle, WA 98195-4350, USA\\[-0.8ex]
\small \texttt{novik@math.washington.edu}
}

\begin{document}
%%%%%%%%%%%%%%%%%%%%%%%%%%%%%%%%%%%%%%%%%%%%
%%%%%%%%%%%%%%%%%%%%%%%%%%%%%%%%%%%%%%%%%%%%
\maketitle

\begin{abstract} We resolve a conjecture of Kalai asserting that the $g_2$-number of any simplicial complex $\Delta$ that represents a connected normal pseudomanifold of dimension $d\geq 3$ is at least as large as ${d+2 \choose 2}m(\Delta)$, where $m(\Delta)$ denotes the minimum number of generators of the fundamental group of $\Delta$. Furthermore, we prove that a weaker bound,  $h_2(\Delta)\geq {d+1 \choose 2}m(\Delta)$, applies to any $d$-dimensional pure simplicial poset $\Delta$ all of whose faces of co-dimension $\geq 2$ have connected links. This generalizes a result of Klee. Finally, for a pure relative simplicial poset $\Psi$ all of whose vertex links satisfy Serre's condition $(S_r)$, we establish lower bounds on $h_1(\Psi),\ldots,h_r(\Psi)$ in terms of the $\mu$-numbers introduced by Bagchi and Datta. 
\end{abstract}

%%%%%%%%%%%%%%%%%%%%%%%
%%%%%%%%%%%%%%%%%%%%%%%
\section{Introduction}
Our starting point is the lower bound theorem of Barnette \cite{Barnette-73}, Kalai \cite{Kalai-87}, and Fogelsanger \cite{Fogelsanger-88} asserting that in the class of all simplicial complexes representing connected normal pseudomanifolds without boundary of dimension $d\geq 2$ and with $n$ vertices, the boundary complex of a stacked $(d+1)$-dimensional polytope simultaneously minimizes all the face numbers. In the last decade or so, a lot of effort went into strengthening these bounds by taking into consideration the topology of the complex in question, see, for instance, \cite{Murai-15,Murai-Nevo-14,Murai-Novik-16,Novik-Swartz-09:Socles,Novik-Swartz-12,Swartz-09}. While the theorems proved and techniques developed significantly increased our understanding of how the Betti numbers of a topological space $M$ affect possible face numbers of triangulations of $M$,  the relationship between the face numbers and other topological invariants, for instance, the fundamental group, remained virtually unknown. The main goal of this note is to use the $\mu$-numbers introduced by Bagchi and Datta \cite{Bagchi-Datta-14} to establish bounds on the face numbers in terms of the minimum number of generators of the fundamental group.

Let $\Delta$ be a $(d-1)$-dimensional simplicial complex.
We denote by $f_i(\Delta)$ the number of $i$-dimensional faces of $\Delta$ for $i=-1,0,1,\dots,d-1$, with $f_{-1}(\Delta)=1$ accounting for the empty face.
The $f$-vector of $\Delta$ is the vector $f(\Delta)=(f_{-1}(\Delta),f_0(\Delta),\dots,f_{d-1}(\Delta))$. The $h$-vector of $\Delta$, $h(\Delta)=(h_0(\Delta),h_1(\Delta),\dots,h_d(\Delta))$, is defined by $\sum_{i=0}^d h_i(\Delta) x^{d-i}=\sum_{j=0}^{d} f_{j-1}(\Delta) (x-1)^{d-j}$; equivalently,
$h_i(\Delta)=\sum_{j=0}^i (-1)^{i-j} {d-j \choose i-j} f_{j-1}(\Delta)$
for $i=0,1,\dots,d$. The consecutive differences, $h_i(\Delta)-h_{i-1}(\Delta)$ are known as the $g$-numbers. Specifically, we will be interested in $g_2:=h_2-h_1$.
When $\Delta$ is connected, we denote by $\pi_1(\Delta)$ the fundamental group of (the geometric realization of) $\Delta$ and by $m(\Delta)$ the minimum number of generators of $\pi_1(\Delta)$.

Kalai \cite{Kalai-87} conjectured that any triangulation $\Delta$ of a connected closed manifold of dimension $d \geq 3$ satisfies
\begin{align}
\label{1-1}
g_2(\Delta)\geq {d+2 \choose 2} b_1(\Delta),
\end{align}
where $b_1(\Delta)$ is the first Betti number of $\Delta$ (computed with coefficients in a field). In fact, he conjectured that under the same assumptions the following stronger inequality holds
\begin{align}
\label{1-2}
g_2(\Delta)\geq {d+2 \choose 2} m(\Delta).
\end{align}
The second inequality is stronger than the first one since, as follows from the Hurewicz theorem, $m(\Delta)\geq b_1(\Delta)$ for any $\Delta$. Moreover, $m(\Delta)$ can be (arbitrarily) larger than $b_1(\Delta)$: there exist manifolds with a perfect fundamental group, and hence vanishing first homology. %For instance, the Poincar\'e sphere has vanishing first homology (over $\Z$) on one hand, and the fundamental group of order 120 with $m=2$ on the other.

Inequality \eqref{1-1} was established in \cite{Novik-Swartz-09:Socles} (for oriented manifolds) and in \cite{Murai-15} (for all manifolds). One of the main results of this paper is the proof of the other inequality:

\begin{theorem}
\label{main1} Let $\Delta$ be a simplicial complex of dimension $d \geq 3$.
Assume further that $\Delta$ is a connected normal pseudomanifold. Then
$$g_2(\Delta)\geq {d+2 \choose 2} m(\Delta).$$ Moreover, if $d \geq 4$, then
$g_2(\Delta) = {d+2 \choose 2} m(\Delta)$ if and only if $\Delta$ is a stacked manifold.
\end{theorem}

Our proof of Theorem \ref{main1} is based on studying the $\mu$-numbers introduced by Bagchi and Datta \cite{Bagchi-Datta-14}, and, specifically, on the following result verified in \cite{Murai-15} (see the proof of Theorem 5.3 there). We postpone the definition of the $\mu$-numbers as well as several other definitions until the next section, and for now merely mention that the $\mu$-numbers satisfy the Morse inequalities; in particular, $\mu_1-\mu_0\geq b_1-b_0$.

\begin{theorem}
\label{LBT:mu} 
Let $\Delta$ be a simplicial complex of dimension $d \geq 3$.
Assume further that $\Delta$ is a connected normal pseudomanifold. Then
$$g_2(\Delta) \geq {d+2 \choose 2} (\mu_1(\Delta)-\mu_0(\Delta)+1)\geq {d+2 \choose 2}b_1(\Delta).$$
Moreover, if $d \geq 4$, then
$g_2(\Delta) = {d+2 \choose 2} (\mu_1(\Delta)-\mu_0(\Delta)+1)$ if and only if $g_2(\Delta) ={d+2 \choose 2}b_1(\Delta)$, which, in turn, happens if and only if $\Delta$ is a stacked manifold.
\end{theorem}

To prove Theorem \ref{main1}, we show that \textit{any} connected simplicial complex (and, even more generally, simplicial poset) $\Delta$ satisfies
\begin{align} \label{mu-vs-m}
\mu_1(\Delta)-\mu_0(\Delta)+1 \geq m(\Delta).
\end{align} 
In other words, the $\mu$-numbers appear to provide much finer information about the complex than the Betti numbers do. This inequality also leads to the following generalization of a result of Klee \cite{Klee-09}. (Klee proved the same statement but in the special case of \textit{balanced} simplicial posets.)

\begin{theorem}
\label{main2}
Let $\Delta$ be a pure connected simplicial complex (or, more generally, a simplicial poset) of dimension $d \geq 2$. Assume further that all faces of $\Delta$ of dimension $\leq d-2$ have connected links. Then
$$h_2(\Delta) \geq {d+1 \choose 2} m(\Delta).$$
\end{theorem}

In order to prove Theorem \ref{main2}, we, in fact, establish the following more general result: 
\begin{theorem}
\label{main3}
Let $r \leq d$ and let $(\Delta,\Gamma)$ be a pure relative simplicial complex (or, more generally, a pure relative simplicial poset) of dimension $d$.
If for each vertex $v$ of $\Delta$, the link $(\lk(v,\Delta),\lk(v,\Gamma))$ satisfies Serre's condition $(S_r)$, then
$$h_i(\Delta,\Gamma) \geq {d+1 \choose i} \left( \sum_{j=1}^i (-1)^{i-j} \mu_{j-1}(\Delta,\Gamma;\KK) +(-1)^i f_{-1}(\Delta,\Gamma) \right) \quad \mbox{for all } i \leq r.$$
In particular, if $(\Delta,\Gamma)$ is Buchsbaum then these inequalities hold for all $i\leq d$.
\end{theorem}
Theorem \ref{main2} is then an immediate consequence of Inequality \eqref{mu-vs-m} along with the ($r=2$ \& $\Gamma= \emptyset$)-case of Theorem \ref{main3}. %Another consequence of Theorem \ref{main3} worth pointing out is that if $(\Delta,\Gamma)$ is Buchsbaum, then Inequalities \eqref{1-3} hold for all $i$.

The rest of the paper is organized as follows.
In Section 2, we review basic results and definitions pertaining to simplicial complexes, including the definition of their $\mu$-numbers. Section 3 is devoted to the proof of  Theorem \ref{main1}. In Section 4 we verify Theorems \ref{main2} and \ref{main3}  for the case of simplicial complexes. Then,
in Section 5, we review the notion of simplicial posets and explain how to extend Theorems \ref{main2} and \ref{main3} to this generality.

\section{Preliminaries}
The goal of this section is to review some basics of simplicial complexes and their $\mu$-numbers.

\subsection{Simplicial complexes}
A \textbf{simplicial complex}  $\Delta$ on $V$ is a collection of subsets of $V$ that is closed under inclusion. A relative simplicial complex $\Psi$ on $V$ is a collection of subsets of $V$ with the property that there exist simplicial complexes $\Delta \supseteq \Gamma$ such that $\Psi=\Delta \setminus \Gamma$.
We identify such a pair of simplicial complexes $(\Delta,\Gamma)$ with the relative simplicial complex $\Delta \setminus \Gamma$.
A simplicial complex $\Delta$ will also be identified with $(\Delta,\emptyset)$.
The elements of $\Delta \setminus \Gamma$ are called \textbf{faces} of $(\Delta,\Gamma)$, the $1$-element faces are called \textbf{vertices}, and the maximal faces (under inclusion) are called \textbf{facets}.
The \textbf{dimension} of a face $F \in \Delta \setminus \Gamma$ is the cardinality of $F$ minus one and the dimension of $(\Delta,\Gamma)$ is the maximum dimension of its faces.
A relative simplicial complex is said to be \textbf{pure} if all of its facets have the same dimension.

Let $(\Delta,\Gamma)$ be a $(d-1)$-dimensional relative simplicial complex.
We let $f_i(\Delta,\Gamma)$ denote the number of $i$-dimensional faces of $(\Delta,\Gamma)$; in particular,
$f_{-1}(\Delta,\Gamma)$ is $1$ if $\emptyset \in \Delta \setminus \Gamma$ and $0$ if $\emptyset \not \in \Delta \setminus \Gamma$.
The components of the \textbf{$h$-vector} $h(\Delta,\Gamma)=(h_0(\Delta,\Gamma),\dots,h_d(\Delta,\Gamma))$ of $(\Delta,\Gamma)$ are defined by
$$h_i(\Delta,\Gamma)= \sum_{j=0}^i (-1)^{i-j} {d-j \choose i-j} f_{j-1}(\Delta, \Gamma) \quad \mbox{ for $i=0,1,\ldots,d$}.$$
We denote by $H_i(\Delta,\Gamma;\KK)$ the $i$th \textbf{homology} of the pair $(\Delta,\Gamma)$ computed with coefficients in a field $\KK$, and by $b_i(\Delta,\Gamma;\KK)$ the $\KK$-dimension of $H_i(\Delta,\Gamma;\KK)$. Similarly, $\tilde H_i(\Delta,\Gamma;\KK)$ and $\tilde b_i(\Delta,\Gamma;\KK)$ stand for the $i$th \textbf{reduced homology} of $(\Delta,\Gamma)$ and the $\KK$-dimension of $\tilde H_i(\Delta,\Gamma;\KK)$, respectively.

A simplicial complex $\Delta$ on $V$ gives rise to several new simplicial complexes. 
The \textbf{link} of a face $F$ in $\Delta$ is
$$\lk(F,\Delta)=\{G \in \Delta: F \cup G \in \Delta,\ F \cap G = \emptyset\}.$$ 
(We also define $\lk(F,\Delta)=\emptyset$ if $F$ is not a face of $\Delta$.)
The \textbf{induced subcomplex} of $\Delta$ on $W \subseteq V$ is
$$\Delta_W=\{ F \in \Delta: F \subseteq W\}.$$
If $v$ is not an element of $V$, then the \textbf{cone} over $\Delta$ with apex $v$ is the simplicial complex 
$$v*\Delta=\Delta \cup \{ \{v \} \cup F: F \in \Delta\}.$$

A $d$-dimensional simplicial complex $\Delta$ is a \textbf{normal pseudomanifold} (without boundary) if (1) it is pure, (2) each $(d-1)$-face of $\Delta$ is contained in exactly two facets of $\Delta$, and (3) the link of each non-empty face of dimension at most $d-2$ is connected. In particular, a triangulation of any closed manifold is a normal pseudomanifold.

A $d$-dimensional \textbf{stacked manifold} is a simplicial complex obtained by starting with several disjoint boundary complexes of the $(d+1)$-dimensional simplex and repeatedly forming connected sums and/or handle
additions. In particular, connected stacked manifolds are homeomorphic to spheres or to  connected sums of sphere bundles over the circle. Thus, when $d\geq 3$, the first Betti number of a $d$-dimensional stacked manifold $\Delta$ (computed over any field) coincides with $m(\Delta)$. This is the only property of stacked manifolds that will be used in this paper.

\subsection{The $\mu$-numbers}
We adopt the following definition of the $\mu$-numbers (cf., \cite[Section 4]{Murai-Novik-16}).
For a finite set $V$, let $\Sm_V$ be the collection of all linear orderings of the elements of $V$.
If $(\Delta,\Gamma)$ is a relative simplicial complex on $V$ with $|V|=n$ and $\varsigma=(v_1,\dots,v_n) \in \Sm_V$,
define
\begin{equation} \label{mu-def}
\mu_i^\varsigma(\Delta,\Gamma;\KK)= \sum_{k=1}^n \tilde b_{i-1}\big(\lk(v_k,\Delta_{\{v_1,\dots,v_{k}\}}),\lk(v_k, \Gamma_{\{v_1,\dots,v_{k}\}});\KK\big) \quad \mbox{for $0\leq i \leq \dim(\Delta,\Gamma)$}
\end{equation}
and
$$\mu_i(\Delta,\Gamma;\KK)= \frac 1 {n!} \sum_{\varsigma \in \Sm_V} \mu_i^\varsigma(\Delta,\Gamma;\KK) \quad \mbox{for $0\leq i \leq \dim(\Delta,\Gamma)$}.$$
When $i \leq 1$, we write $\mu_i^\varsigma(\Delta,\Gamma)$ instead of $\mu_i^\varsigma(\Delta,\Gamma;\KK)$ and $\mu_i(\Delta,\Gamma)$ instead of $\mu_i(\Delta,\Gamma;\KK)$
since these $\mu$-numbers do not depend on $\KK$.

The $\mu^\varsigma$-numbers were essentially introduced by Brehm and K\"uhnel \cite{Brehm-Kuhnel-87}
while the $\mu$-numbers are due to Bagchi and Datta \cite{Bagchi-Datta-14}.
(Note that our $\mu_0$ is slightly different than $\mu_0$ of \cite{Bagchi-Datta-14}.)
As can be seen from the above definition, these numbers are related to a version of the Morse theory for simplicial complexes
(cf.~\cite[Remark 2.11]{Bagchi-Datta-14}). In particular, they satisfy the following inequalities known as the Morse inequalities (see \cite[Lemma 4.2]{Murai-Novik-16}).

\begin{lemma}
\label{2.1}
Let $(\Delta,\Gamma)$ be a relative simplicial complex on $V$ and let $\varsigma \in \Sm_V$. Then
$$\sum_{j=0}^i (-1)^{i-j} b_j (\Delta,\Gamma;\KK) \leq \sum_{j=0}^i (-1)^{i-j} \mu^\varsigma_j(\Delta,\Gamma;\KK) \quad \mbox{ for all }i \geq 0.$$
\end{lemma}

Note that since $\mu_j(\Delta,\Gamma;\KK)$ is the average of $\mu^\varsigma_j(\Delta,\Gamma;\KK)$ over all orderings $\varsigma\in \Sm_V$,
the above inequalities continue to hold with the $\mu^\varsigma$-numbers replaced by the $\mu$-numbers.

\subsection{Stanley--Reisner rings and modules}

We now turn to reviewing the Stanley--Reisner modules and their connection to the $\mu$-numbers.
Let $\Delta$ be a simplicial complex on $V$ and let $S=\KK[x_v:v \in V]$ be the polynomial ring over a field $\KK$ with $\deg x_v =1$ for all $v \in V$.
The ideal
$$I_\Delta=(x_{v_1} x_{v_2} \cdots x_{v_k}: \{v_1,v_2,\dots,v_k\} \not \in \Delta) \subseteq S$$ is called the \textbf{Stanley--Reisner ideal} of $\Delta$ 
and the ring $\KK [\Delta]= S/I_\Delta$ is called the \textbf{Stanley--Reisner ring} of $\Delta$. 
Similarly, if $(\Delta,\Gamma)$ is a relative simplicial complex, then the \textbf{Stanley--Reisner module} of $(\Delta,\Gamma)$ is the $S$-module
$$\KK[\Delta,\Gamma] = I_\Gamma/I_\Delta.$$
%(Note that we define $I_\Gamma$ as an ideal of $S$.)

When $M$ is a finitely generated graded $S$-module,
the numbers
$$\beta_{i,j}(M)=\dim_\KK \Tor_i^S(M,\KK)_j$$
are called the \textbf{graded Betti numbers} of $M$.
Here $N_k$ denotes the degree-$k$ graded component of a graded $S$-module $N$.
For a relative simplicial complex $(\Delta,\Gamma)$ on $V$ with $|V|=n$,
we define the $\tilde \sigma$-numbers of  $(\Delta,\Gamma)$ by
$$\tilde \sigma _{i-1}(\Delta,\Gamma;\KK)= \frac 1 {n+1} \sum_{k=0}^n \frac 1 {{n \choose k}}
\beta_{k-i,k}(\KK[\Delta,\Gamma]).$$
We will rely on the following interpretation of  the $\mu$-numbers in terms of the $\tilde \sigma$-numbers, see \cite[Sections 4 \& 6]{Murai-Novik-16}. This interpretation is a consequence of Hochster's formula for the graded Betti numbers of Stanley--Reisner modules.

\begin{lemma}
\label{2.2}
Let $(\Delta,\Gamma)$ be a relative simplicial complex on $V$. Then
$$\mu_i(\Delta,\Gamma;\KK)=\sum_{v \in V} \tilde \sigma_{i-1}\big(\lk(v,\Delta), \lk(v,\Gamma);\KK\big) \quad \mbox{for } i \geq 0.$$
\end{lemma}

\section{The $\mu$-numbers and the fundamental group}

In this section we prove Theorem \ref{main1}. This will require a few lemmas.

\begin{lemma}
\label{3.1}
Let $\Delta$ and $\Gamma$ be simplicial complexes. Assume further that $\Delta$ is connected, $\Gamma$ is contractible, and  
$\Delta \cap \Gamma$ has exactly two connected components. Then
$$m(\Delta \cup \Gamma) \leq m(\Delta)+1.$$
\end{lemma}

\begin{proof}
Denote the two connected components of $\Delta \cap \Gamma$ by
$\Sigma_1$ and $\Sigma_2$.
Since $\Delta$ is connected, there is a path 
$$\gamma=\{\{v_1,v_2\},\{v_2,v_3\},\dots,\{v_{s-1},v_s\}\} \subset \Delta$$
such that $v_1 \in \Sigma_1$, $v_s\in \Sigma_2$, and $v_2,\dots,v_{s-1}$ do not belong to $\Gamma$.
Let $\Gamma' = \Gamma \cup \gamma$.
Since $\Gamma$ is contractible, $\Gamma'$ is homotopy equivalent to $\mathbb S^1$.
Furthermore, $\Delta \cap \Gamma' = \Sigma_1 \cup \Sigma_2 \cup \gamma$ is connected.
Therefore, by the Seifert--van Kampen theorem, the fundamental group $\pi_1(\Delta \cup \Gamma')$ is a quotient of the free product $\pi_1(\Delta) * \pi_1(\Gamma')=\pi_1(\Delta)* \ZZ$. 
Since $\Delta \cup \Gamma'=\Delta \cup \Gamma$,
this fact implies the desired inequality.
\end{proof}

\begin{lemma}
\label{3.2}
Let $\Delta$ be a connected simplicial complex, let $\Gamma_1,\Gamma_2,\dots,\Gamma_s$ be several connected and pairwise disjoint subcomplexes of $\Delta$, and let $v$ be an element that is not in $\Delta$.
Then
$$m\big(\Delta \cup (v*(\Gamma_1 \cup \cdots \cup \Gamma_s))\big) \leq m(\Delta)+s-1.$$
\end{lemma}

\begin{proof} Assume first that $s\geq 2$.
Since $v * \Gamma_s$ is contractible and since its intersection with $\Delta \cup (v*(\Gamma_1 \cup \cdots \cup \Gamma_{s-1}))$ consists of two connected components, namely, the vertex $v$ and the complex $\Gamma_s$, we obtain from Lemma \ref{3.1} that
$$m\big(\Delta \cup (v*(\Gamma_1 \cup \cdots \cup \Gamma_s))\big)
\leq m\big(\Delta \cup (v*(\Gamma_1 \cup \cdots \cup \Gamma_{s-1}))\big) +1.$$ Thus, to complete the proof, it suffices to show that $m(\Delta \cup (v*\Gamma_1)) \leq m(\Delta)$. Indeed, since $\Delta \cap (v*\Gamma_1)=\Gamma_1$ is connected and $v * \Gamma_1$ is contractible, this inequality is a consequence of the Seifert--van Kampen theorem. The statement follows.
\end{proof}

For the rest of this section, it is convenient to extend the definition of $m(\Delta)$ to the class of all simplicial complexes (including the disconnected ones) by letting $m(\Delta)$ be the sum of the $m$-numbers of the connected components of $\Delta$. With this definition in hand, we can state the following result that together with Theorem \ref{LBT:mu} immediately implies the inequality part of Theorem \ref{main1}.

\begin{theorem}
\label{3.3}
Let $\Delta$ be a simplicial complex on $V$ and let $\varsigma=(v_1,\dots,v_n) \in \Sm_V$.
Then
$$\mu_1^\varsigma(\Delta) -\mu_0^\varsigma(\Delta) \geq m(\Delta)-b_0(\Delta).$$
In particular, if $\Delta$ is connected, then $\mu_1(\Delta) -\mu_0(\Delta)+1 \geq m(\Delta)$.
\end{theorem}

\begin{proof} The ``in particular" part follows from the first part since $\mu_1(\Delta) -\mu_0(\Delta)+1$ is the average of $\mu^\varsigma_1(\Delta) -\mu^\varsigma_0(\Delta)+1$ over all orderings $\varsigma$.
To prove the first part, we use induction on $|V|=n$.
The desired inequality does hold when $\Delta=\{ \emptyset\}$ (with both sides equal to $0$).
Thus we may assume that $n \geq 1$.

Let $\varsigma'=(v_1,\dots,v_{n-1})$ and $\Delta'=\Delta_{\{v_1,\dots,v_{n-1}\}}$.
Then $\Delta=\Delta' \cup (v_n* \lk(v_n,\Delta))$ and
$$\mu_1^\varsigma(\Delta)-\mu_0^\varsigma(\Delta) =\mu_1^{\varsigma'}(\Delta')-\mu_0^{\varsigma'}(\Delta')+ \big(\tilde b_0(\lk(v_n,\Delta))-\tilde b_{-1}(\lk(v_n,\Delta))\big).$$
%and $$\Delta=\Delta' \cup (v_n* \lk_\Delta(v_n)).$$
Therefore, by the induction hypothesis, to complete the proof,
it suffices to show that
\begin{align}
\label{3-1}
m(\Delta')-b_0(\Delta') + \big(\tilde b_0(\lk(v_n,\Delta)) - \tilde b_{-1}(\lk(v_n,\Delta))\big)
\geq m(\Delta)-b_0(\Delta).
\end{align}
If $\lk(v_n,\Delta)=\{ \emptyset\}$,
then $\Delta$ is the disjoint union of $\Delta'$ and the vertex $v_n$. Hence $m(\Delta)=m(\Delta')$ and $b_0(\Delta)=b_0(\Delta')+1$, while $\tilde b_0(\lk(v_n,\Delta)) - \tilde b_{-1}(\lk(v_n,\Delta))=0-1=-1$, yielding that \eqref{3-1} holds as equality in this case.

Suppose $\lk(v_n,\Delta) \ne \{\emptyset\}$.
Let $t=b_0(\Delta')$ and let $\Gamma_1,\dots,\Gamma_t$ be the connected components of $\Delta'$.
Let $\Sigma_i=\Gamma_i \cap \lk(v_n,\Delta)$ for $i=1,2,\dots,t$.
Without loss of generality, assume that the number of connected components of $\Sigma_i$ is $s_i \geq 1$ for $i=1,2,\dots,\ell$ and that $\Sigma_i$ is empty for $i > \ell$.
Then 
\begin{eqnarray}
\label{3-2}
&&\sum_{k=1}^\ell s_i = b_0(\lk(v_n,\Delta))= \tilde b_0(\lk(v_n,\Delta))- \tilde b_{-1}(\lk(v_n,\Delta)) +1,\\
\label{3-3}
&&b_0(\Delta')-b_0(\Delta)=\ell-1, \\
\label{3-4}
&&m\big(\Gamma_i \cup (v_n* \lk(v_n,\Delta))\big)\leq m(\Gamma_i) + s_i-1 \quad \mbox{for $i=1,2,\dots,\ell$},
\end{eqnarray}
where the last inequality follows from Lemma \ref{3.2}. 
Since
\begin{align*}
\Delta&=\Delta' \cup (v_n* \lk(v_n,\Delta))\\
&= \big(\Gamma_1 \cup (v_n*\Sigma_1)\big) \cup \cdots \cup \big(\Gamma_\ell \cup (v_n*\Sigma_\ell)\big) 
\cup \Gamma_{\ell+1} \cup \cdots \cup \Gamma_t,
\end{align*}
and since $\Gamma_k\cup (v_n* \Sigma_k)$ and $\Gamma_{k'}\cup (v_n*\Sigma_{k'})$ intersect in the same single vertex $v_n$ for all $1\le k < k'\le \ell$, we infer from 
\eqref{3-2}--\eqref{3-4} (and the Seifert-van Kampen theorem) that
\begin{align*}
m(\Delta) & = m\big(\Gamma_1 \cup (v_n*\Sigma_1)\big) + \cdots + m\big(\Gamma_\ell \cup (v_n*\Sigma_\ell)\big) + m(\Gamma_{\ell+1})+ \cdots + m(\Gamma_t)\\
&\leq \sum_{k=1}^t m(\Gamma_k) + \sum_{k=1}^\ell (s_k-1)\\
&= m(\Delta') + \big(\tilde b_0(\lk(v_n,\Delta))-\tilde b_{-1}(\lk(v_n,\Delta))\big) +\big(b_0(\Delta)-b_0(\Delta')\big),
\end{align*}
and  Inequality \eqref{3-1} follows.
\end{proof}

To complete the poof of Theorem \ref{main1}, it only remains to discuss the case of equality. Since $g_2(\Delta)\geq {d+2 \choose 2}(\mu_1(\Delta)-\mu_0(\Delta)+1)\geq {d+2 \choose 2}m(\Delta)$, the statement follows easily from the equality case of Theorem \ref{LBT:mu} along with the observation that for a stacked manifold $\Delta$ of dimension $\geq 3$, $m(\Delta)=b_1(\Delta)$ (see the end of Section 2.1). 
\hfill$\square$\medskip

It is known that if $\Delta$ is a $3$-dimensional connected normal pseudomanifold, then $g_2(\Delta)=10 b_1(\Delta)$ if and only if $\Delta$ is a stacked manifold. However, the class of $3$-dimensional normal pseudomanifolds that satisfy $g_2(\Delta)=10 (\mu_1(\Delta)-\mu_0(\Delta)+1)$ is larger: it consists of simplicial manifolds all of whose vertex links are stacked spheres, and, for instance, includes the boundary complexes of $4$-dimensional cyclic polytopes. (See the proof of \cite[Theorem 5.3]{Murai-15} for both statements.) It would be interesting to understand which $3$-dimensional connected normal pseudomanifolds satisfy the inequality $g_2(\Delta)\geq 10m(\Delta)$ of Theorem \ref{main1} as equality.

It is also worth mentioning that the proof of Theorem \ref{LBT:mu} given in \cite{Murai-15}, in fact, shows that the inequality part of this theorem applies to a larger class of complexes; specifically, it holds
for an arbitrary $d$-dimensional pure simplicial complex all of whose vertex links are generically $d$-rigid.
As Buchsbaum* simplicial complexes (introduced in \cite{Athanasiadis-Volkmar-12}) of dimension $d \geq 3$ have generically $d$-rigid vertex links (see \cite[Theorem 4.1]{Athanasiadis-Volkmar-12} for an even stronger result), we conclude that such complexes satisfy the inequality of Theorem \ref{LBT:mu} and hence also of Theorem \ref{main1}:

\begin{theorem}
Let $\Delta$ be a connected simplicial complex of dimension $d \geq 3$. Assume further that $\Delta$ is Buchsbaum*. 
Then
$g_2(\Delta) \geq {d+2 \choose 2} m(\Delta).$
\end{theorem}

\section{Lower bounds on the $h$-numbers}

In this section we consider pure simplicial complexes all of whose vertex links satisfy Serre's condition $(S_r)$; we establish lower bounds on the $h$-numbers of such complexes in terms of their $\mu$-numbers.
A relative simplicial complex $(\Delta,\Gamma)$ satisfies \textbf{Serre's condition} ($S_r$) (over $\KK$) if for every face $F \in \Delta$ (including the empty face),
$$\tilde H_i\big(\lk(F,\Delta),\lk(F,\Gamma);\KK\big)=0
\quad \mbox{ for all } i< \min\left\{ r-1,\dim\!\big(\lk(F,\Delta), \lk(F,\Gamma)\big)\right\}.$$

We recall some basic facts on complexes satisfying Serre's conditions. These conditions arise from Serre's $(S_r)$ condition in commutative algebra (see \cite[Section 2.1]{BH} for the definition). It is known that if $r \geq 2$, then a finitely generated graded module $M$ over a polynomial ring satisfies Serre's condition $(S_r)$ if and only if its deficiency module $K_M^j$ (that is, the Matlis dual of the $j$th local cohomology module of $M$) has Krull dimension $\leq j -r$  for all $0 \leq j <d$ (see \cite[Lemma 3.2.1]{Sc82}).
Consequently, for $r \geq 2$,
it follows from Hochster's formula for local cohomology modules \cite[Theorem 1.8]{Adiprasito-Sanyal} that a $(d-1)$-dimensional relative simplicial complex
$(\Delta,\Gamma)$ satisfies $(S_r)$ if and only if
$\KK[\Delta,\Gamma]$ satisfies Serre's condition ($S_r$). %{\color{blue}(I will look for a reference)} 
Thus, a $(d-1)$-dimensional relative complex $(\Delta,\Gamma)$ satisfies $(S_d)$ if and only if $(\Delta,\Gamma)$ is Cohen--Macaulay;
furthermore, a $(d-1)$-dimensional simplicial complex $\Delta$ satisfies $(S_2)$ if and only if it is pure and for every face $F \in \Delta$ of dimension $\leq d-3$, the link of $F$, $\lk(F,\Delta)$, is connected.

The following result was proved in  \cite{Murai-Terai}.

\begin{theorem}
\label{4.1}
Let $(\Delta,\Gamma)$ be a $(d-1)$-dimensional relative simplicial complex.
If $(\Delta,\Gamma)$ satisfies $(S_r)$ over $\KK$ then there is a sequence of linear forms $\theta_1,\dots,\theta_d$ such that
\begin{itemize}
\item[(i)] the multiplication map
$$\times \theta_i : \big(\KK[\Delta,\Gamma]/(\theta_1,\dots,\theta_{i-1} )\KK[\Delta,\Gamma]\big)_{k-1} 
\to \big(\KK[\Delta,\Gamma]/(\theta_1,\dots,\theta_{i-1} )\KK[\Delta,\Gamma]\big)_k$$
is injective for all $i=1,2,\dots,d$ and all $k \leq r$, and
\item[(ii)]
$\dim_\KK\big(\KK[\Delta,\Gamma]/(\theta_1,\dots,\theta_{d} )\KK[\Delta,\Gamma]\big)_i=h_i(\Delta,\Gamma)$ for all $i \leq r$.
\end{itemize}
\end{theorem}

While Theorem \ref{4.1} was verified in \cite{Murai-Terai} only for the case of Stanley--Reisner rings, it also holds in the generality of Stanley--Reisner modules. Indeed,  the proofs given in \cite{Murai-Terai} apply to an arbitrary squarefree module (a notion introduced and studied by Yanagawa in \cite{Ya}), while by \cite[Lemma 2.3]{Ya} all Stanley--Reisner modules are squarefree. (See also the proof of \cite[Theorem 5.6]{Yanagawa-11}.)

Since the proof of the main result of this section follows the same outline as the proof  of \cite[Theorem 6.5]{Murai-Novik-16}, instead of providing complete details we only sketch the main ideas. The following result is \cite[Lemma 5.5]{Murai-Novik-16}.

\begin{lemma}
\label{4.2}
Let $S=\KK[x_1,\dots,x_n]$, let $M$ be a finitely generated graded $S$-module, and let $\theta_1,\dots,\theta_d$ be linear forms.
If
$$\times \theta_i : \big(M/(\theta_1,\dots,\theta_{i-1})M\big)_{k-1} \to \big(M/(\theta_1,\dots,\theta_{i-1})M\big)_{k} 
$$
is injective for all $i=1,2,\dots,d$ and all $k \leq r$, then
$$\sum_{k\geq 0} (-1)^k \beta_{i+k,i+\ell}(M) \leq \sum_{k \geq 0} (-1)^k {n-d \choose i+k}\dim_\KK (M/(\theta_1,\dots,\theta_d)M)_{\ell-k} $$
for all $i \geq 0$ and $\ell \leq r-1$.
\end{lemma}

Theorem \ref{4.1} asserts that if $(\Delta,\Gamma)$ satisfies $(S_r)$ then $\KK[\Delta,\Gamma]$ satisfies the assumptions of Lemma \ref{4.2}.
Substituting the upper bounds of Lemma \ref{4.2} into the definition of $\tilde \sigma$-numbers %, $\tilde \sigma_{j-1}(\Delta,\Gamma;\KK)$, 
leads to the following upper bounds on their alternating sums.
(We omit the computation as it is exactly the same as in the proof of \cite[Proposition 6.2]{Murai-Novik-16} but with $d+1$ replaced by $d$.)

\begin{lemma}
\label{4.3}
If a $(d-1)$-dimensional relative simplicial complex $(\Delta, \Gamma)$ satisfies Serre's condition $(S_r)$ over $\KK$ then
$$\sum_{j=0}^i (-1)^{i-j} \tilde \sigma_{j-1}(\Delta,\Gamma;\KK) \leq \frac 1 {d+1} \sum_{j=0}^i (-1)^{i-j} \frac {h_j(\Delta,\Gamma)} {{d \choose j}} \mbox{ for } i \leq r-1.$$
\end{lemma}

We are now in a position to prove Theorem \ref{main3} for the class of $d$-dimensional pure relative simplicial \textit{complexes} all of whose vertex links satisfy $(S_r)$. Recall that a pure relative simplicial complex $(\Delta,\Gamma)$ is called \textbf{Buchsbaum} (over $\KK$) if for every vertex $v$ of $\Delta$, $(\lk(v,\Delta),\lk(v,\Gamma))$ is Cohen--Macaulay (over $\KK$). In other words, a $d$-dimensional $(\Delta,\Gamma)$ is Buchsbaum if and only if it is pure and all vertex links of $(\Delta,\Gamma)$  satisfy $(S_d)$. 
%Since, for $(d-1)$-dimensional complexes, Cohen--Macaulayness is equivalent to $(S_d)$ where $\dim(\Delta,\Gamma)=d-1$

%\begin{theorem}
%\label{4.4}
%Let $r \leq d$ and let $(\Delta,\Gamma)$ be a pure relative simplicial complex of dimension $d$.
%If for each vertex $v$ of $\Delta$, the link $(\lk_\Delta(v),\lk_\Gamma(v))$ satisfies Serre's condition $(S_r)$, then
%$$h_i(\Delta) \geq {d+1 \choose i} \left( \sum_{j=1}^i (-1)^{i-j} \mu_{j-1}(\Delta,\Gamma;\KK) +(-1)^i f_{-1}(\Delta,\Gamma) \right) \quad \mbox{for all } i \leq r.$$
%\end{theorem}

\smallskip\noindent {\it Proof of Theorem \ref{main3}: \ }
Let $V$ be the vertex set of $\Delta$.
Since $(\Delta,\Gamma)$ is pure,
it follows from \cite[Proposition 2.3]{Swartz-04/05} that
$$i h_i(\Delta,\Gamma) +(d-i+2) h_{i-1}(\Delta,\Gamma)= \sum_{v \in V} h_{i-1}\big(\lk(v,\Delta),\lk(v,\Gamma)\big)$$
for $i=1,2,\dots,d+1$.
This result together with Lemmas \ref{2.2} and  \ref{4.3} yields that for all $i\leq r$,
\begin{align*}
\sum_{j=1}^i (-1)^{i-j} \mu_{j-1}(\Delta,\Gamma;\KK)
&= \sum_{j=1}^i (-1)^{i-j} \left( \sum_{v \in V} \tilde \sigma_{j-2} \big(\lk(v,\Delta),\lk(v,\Gamma);\KK\big) \right)\\
& \leq \sum_{j=1}^i (-1)^{i-j} \frac 1 {(d+1) {d \choose j-1}} \left( \sum_{v \in V} h_{j-1}\big(\lk(v,\Delta),\lk(v,\Gamma) \big) \right)\\
&= \sum_{j=1}^{i} (-1)^{i-j} \frac 1 {(d+1) {d \choose j-1}} \big(j h_j(\Delta,\Gamma)+(d-j+2) h_{j-1}(\Delta,\Gamma)\big)\\
&= \sum_{j=1}^{i} (-1)^{i-j}
\left( \frac {h_j(\Delta,\Gamma)} {{d+1 \choose j}} + \frac {h_{j-1}\big(\lk(v,\Delta),\lk(v,\Gamma)\big)} {{d+1 \choose j-1}} \right)\\
&=
\frac {h_i(\Delta,\Gamma)} {{d+1 \choose i}} - (-1)^i h_0(\Delta,\Gamma).
\end{align*}
Since, by definition, $h_0(\Delta,\Gamma)=f_{-1}(\Delta,\Gamma)$,
the above inequality completes the proof.
\hfill$\square$\medskip
%\end{proof}

As we discussed in the Introduction, Theorem \ref{main2} is an immediate consequence of Inequality \eqref{mu-vs-m} and the ($r=2$ \& $\Gamma= \emptyset$)-case of Theorem \ref{main3}. %with $r=2$ and $\Gamma= \emptyset$ implies Theorem \ref{main2}. Indeed, by Proposition \ref{3.3}, Theorem \ref{main2} is the special case when $r=2$ and $\Gamma= \emptyset$.

%We close this section by one more application of Theorem \ref{4.4}. A pure simplicial complex $(\Delta,\Gamma)$ is said to be \textbf{Buchsbaum} (over $\KK$) if $(\lk_\Delta(v),\lk_\Gamma(v))$ is Cohen--Macaulay (over $\KK$) for any vertex $v$ of $\Delta$. Since Cohen--Macaulay condition is equivalent to $(S_d)$ where $\dim(\Delta,\Gamma)=d-1$, by Lemma \ref{4.4} and the Morse inequality, we get

%\begin{corollary}
%\label{4.5}
%Let $(\Delta,\Gamma)$ be a $d$-dimensional relative simplicial complex.If $(\Delta,\Gamma)$ is Buchsbaum over $\KK$, then $$h_i(\Delta,\Gamma) \geq {d+1 \choose i} \left( \sum_{j=1}^i (-1)^{i-j} \mu_{j-1}(\Delta,\Gamma;\KK)  + (-1)^i f_{-1}(\Delta,\Gamma)\right)$$ for all $i \leq d$. In particular, $h_i(\Delta,\Gamma) \geq {d+1 \choose i} ( \sum_{j=1}^i (-1)^{i-j} \beta_{j-1}(\Delta,\Gamma,\KK))$ for all $i \leq d$.
%\end{corollary}

We close this section with one additional remark. According to Theorem \ref{main3}, if $(\Delta,\Gamma)$ is a $d$-dimensional Buchsbaum relative simplicial complex, then  $$h_i(\Delta,\Gamma) \geq {d+1 \choose i} \left( \sum_{j=1}^i (-1)^{i-j} \mu_{j-1}(\Delta,\Gamma;\KK)  + (-1)^i f_{-1}(\Delta,\Gamma)\right) \quad \mbox{for all $i \leq d$}.$$ These inequalities provide a strengthening of a previously known fact \cite[Theorem 3.4]{Novik-Swartz-09:Socles} that for a Buchsbaum   $(\Delta,\Gamma)$, $h_i(\Delta,\Gamma) \geq {d+1 \choose i}  \sum_{j=1}^i (-1)^{i-j} \tilde b_{j-1}(\Delta,\Gamma;\KK)$ for all $i \leq d$.

\section{Simplicial posets}

The goal of this section is to explain how the proofs of Theorems \ref{main2} and \ref{main3} can be extended to the generality of simplicial posets. This requires a quick review of  the definition of simplicial posets and related notions as well as of the corresponding algebraic background.

A \textbf{simplicial poset} $\Delta=(\Delta,\preceq)$ is a finite poset with a unique minimal element, $\emptyset$, and with the property that for every $F\in\Delta$, the interval $[\emptyset, F]$
is a Boolean lattice \cite{Stanley-91}. Consequently, $\Delta$ is graded and atomic.  Furthermore, it follows from the results of \cite{Bjorner-84} that any simplicial poset $\Delta$ is the face poset of a certain regular CW-complex, $|\Delta|$, all of whose closed cells are simplices. Thus, one can think of a simplicial poset as a collection of simplices glued together in a way that every two simplices intersect along an arbitrary (possibly empty) subcomplex of their boundaries. Another important consequence of \cite{Bjorner-84} is that $|\Delta|$ has a well-defined \textbf{barycentric subdivision}, denoted by $\sd(\Delta)$, which is the simplicial complex isomorphic to the order complex of $\Delta-\{\emptyset\}$ and homeomorphic to $|\Delta|$. In particular, the (singular) Betti numbers and the fundamental group of $|\Delta|$ coincide with the simplicial Betti numbers and the fundamental group of $\sd(\Delta)$, respectively. 

From now on we refer to a simplicial poset $\Delta$ and its realization $|\Delta|$ almost interchangeably. The elements of $\Delta$ are called faces and the elements of rank $1$ are called vertices. The dimension of a face $F$ is defined as the rank of the interval $[\emptyset, F]$ minus $1$, and the vertex set of $F$ is defined by $V(F)=\{v\in\Delta \ : \rk(v)=1, v\preceq F\}$.

Since all the faces of a simplicial poset are simplices, many definitions pertaining to simplicial complexes have natural extensions to simplicial posets. For instance, if $\Delta$ is a simplicial poset and $v$ is a vertex of $\Delta$, then the link of $v$ in $\Delta$ is defined by $\lk(v,\Delta)=\{F\in \Delta \ : \ v\preceq F\}$. It is easy to check that $\lk(v,\Delta)$ is  a simplicial poset with minimal element $v$, and that $\sd(\lk(v,\Delta))$ is isomorphic to $\lk(v,\sd(\Delta))$. (However, it is worth pointing out that, in contrast to the setting of simplicial complexes, a vertex link of a simplicial poset $\Delta$ is not naturally a subcomplex of $\Delta$: for instance, consider the link of a vertex in a simplicial poset consisting of two vertices and three edges joining them.) Similarly, for a subset $W$ of vertices of $\Delta$, the restriction of $\Delta$ to $W$ is defined by $\Delta_W=\{F\in\Delta \ : \ V(F)\subseteq W\}$.
%Also, for a vertex $v$ of $\Delta$, we write $\lk_\Delta(v)_W=\{ F \in \Delta\ : \ V(F)\setminus \{v\} \subseteq W\}$.

In analogy with relative simplicial complexes, a \textbf{relative simplicial poset} is a pair of simplicial posets $(\Delta, \Gamma)$ such that $\Gamma\subseteq \Delta$ is a lower ideal of $\Delta$.
%, that is, if $F\in\Gamma$, $G\in\Delta$, and $G\leq F$, then $G\in\Gamma$. 
The dimension of a relative simplicial poset $(\Delta, \Gamma)$ is $\dim (\Delta, \Gamma)=\max\{\dim F \ : \ F\in \Delta \setminus \Gamma\}$. As in the case of relative simplicial complexes, we denote by $f_i=f_{i}(\Delta, \Gamma)$ the number of $i$-dimensional faces of $\Delta \setminus \Gamma$ and  by $(f_{-1},f_0\ldots, f_{\dim(\Delta, \Gamma)})$ the $f$-vector of $(\Delta,\Gamma)$. With the above notions in hand, the definitions from Section 2 of the $h$-vector, the $\mu^\varsigma$-numbers, and the $\mu$-numbers for a relative simplicial complex carry over verbatim to the definitions of the same objects for a relative simplicial poset. 

The following lemma provides a connection between the $\mu$-numbers of a relative simplicial poset $(\Delta,\Gamma)$ and the $\mu$-numbers of its barycentric subdivision --- the relative simplicial complex $(\sd(\Delta), \sd(\Gamma))$. If $\Delta$ is a simplicial poset with vertex set $V$, $|V|=n$, and $\varsigma=(v_1,\ldots,v_n)\in\Sm(V)$, we define an extension of $\varsigma$ to an ordering $\sd(\varsigma)$ of the vertices of $\sd(\Delta)$ by inserting between $v_{k-1}$ and $v_k$ the barycenters of the faces of $\Delta_{\{v_1,\ldots,v_k\}}$ that contain $v_k$, listing them in the increasing order of the dimension of their faces. (Note that the choice of $\sd(\varsigma)$ may not be unique.)

\begin{lemma} \label{mu-properties}
Let $(\Delta, \Gamma)$ be a relative simplicial poset, let $V$ be the vertex set of $\Delta$, and let $\varsigma\in\Sm(V)$. For an arbitrary field $\KK$ and for all $i\geq 0$,
\begin{equation} \label{mu=mu}
\mu_i^{\varsigma}(\Delta, \Gamma; \KK)=\mu_i^{\sd(\varsigma)}(\sd(\Delta), \sd(\Gamma); \KK).
\end{equation}
In particular, the $\mu^\varsigma$-numbers of $(\Delta, \Gamma)$, and hence also the $\mu$-numbers of $(\Delta, \Gamma)$, satisfy the Morse inequalities of Lemma \ref{2.1}; furthermore, if $\Delta$ is connected, then $\mu_1(\Delta)-\mu_0(\Delta)+1 \geq m(\Delta)$.
\end{lemma}

\begin{proof} Let $U$ be the vertex set of $\sd(\Delta)$. When $w\in U$, we denote by $\sd(\Delta)_{\leq w}$ and $\sd(\Gamma)_{\leq w}$ the subcomplexes of $\sd(\Delta)$ and $\sd(\Gamma)$, respectively, induced by $\{u\in U \ : \ u\leq w\}$ --- the initial segment of $U$  in the ordering $\sd(\varsigma)$. As follows from \eqref{mu-def}, to prove \eqref{mu=mu}, it suffices to show that for each vertex $v_k\in V$, the contribution of $v_k$ to both sides of \eqref{mu=mu} is the same, that is,
\[
\tilde b_{i-1}\big(\lk(v_k,\Delta_{\{v_1,\ldots,v_{k}\}}), \lk(v_k,\Gamma_{\{v_1,\ldots,v_{k}\}});\KK\big)= \tilde b_{i-1}\big(\lk(v_k, \sd(\Delta)_{\leq v_k}), \lk(v_k,\sd(\Gamma)_{\leq v_k});\KK\big),
\]
while for each $w\in U\setminus V$, the contribution of $w$ to the right-hand side of \eqref{mu=mu} is zero, that is,
\begin{equation}\label{zero}
\tilde b_{i-1}\big(\lk(w,\sd(\Delta)_{\leq w}), \lk(w,\sd(\Gamma)_{\leq w});\KK\big)=0.\end{equation}
Indeed, the former equation holds since the relative complex $\big(\lk(v_k, \sd(\Delta)_{\leq v_k}), \lk(v_k,\sd(\Gamma)_{\leq v_k})\big)$ is the barycentric subdivision of $\big(\lk(v_k,\Delta_{\{v_1,\ldots,v_{k}\}}), \lk(v_k,\Gamma_{\{v_1,\ldots,v_{k}\}})\big)$. As for the latter equation, suppose that $w$ is the barycenter of a face $F$ of $\Delta$, and that $v_k$ is the maximal vertex of $F$ in the ordering $\varsigma$. If $F$ is not a face of $\Gamma$, then $\lk(w,\sd(\Gamma)_{\leq w})=\emptyset$ while $\lk(w, \sd(\Delta)_{\leq w})$ is the complex obtained by barycentrically subdividing the boundary of the simplex $F$ and then deleting the vertex $v_k$. Thus, in this case $\big(\lk(w,\sd(\Delta)_{\leq w}), \lk(w,\sd(\Gamma)_{\leq w})\big)=\lk(w,\sd(\Delta)_{\leq w})$ is contractible. On the other hand, if $F$ is also a face of $\Gamma$, then so are all the faces of $F$. Hence $\lk(w,\sd(\Delta)_{\leq w})=\lk(w,\sd(\Gamma)_{\leq w})$, that is, $\big(\lk(w,\sd(\Delta)_{\leq w}), \lk(w,\sd(\Gamma)_{\leq w})\big)=\emptyset$. In either case, \eqref{zero} follows.
\end{proof}

We now turn our discussion to analogs of Stanley--Reisner rings and modules (introduced in \cite{Stanley-91}) for simplicial posets. Let $\Delta$ be a simplicial poset. Consider  the polynomial ring $\tilde{S}=\KK[x_F \ : \ F\in\Delta]$ with one variable per each face of $\Delta$. The \textbf{face ideal} of $\Delta$, $J_\Delta$, is the ideal of $\tilde{S}$ generated by the elements of the following form:
\begin{itemize} 
\item $x_Fx_G$ for all pairs of elements $F,G\in\Delta$ that have no common upper bound in $\Delta$;
\item $x_Fx_G - x_{F\wedge G}\sum x_H$ for pairs of $F,G$ incomparable in $\Delta$, where the sum is over the set of all minimal upper bounds of $F$ and $G$. (If $F$ and $G$ have an upper bound $H$, then $F$ and $G$ are elements of $[\emptyset, H]$, which is a Boolean lattice, and so $F\wedge G$ is well defined.)
\end{itemize} 
The ring $\KK[\Delta]=\tilde{S}/J_\Delta$ is called the \textbf{face ring} of $\Delta$. If $(\Delta, \Gamma)$ a simplicial poset, then we define the \textbf{face module} of $(\Delta, \Gamma)$ as $\KK[\Delta,\Gamma]=J_\Gamma/J_\Delta$.

The key that allows us to extend Theorems \ref{main2} and \ref{main3} to the generality of simplicial posets comes from the body of work showing that the face rings and modules of simplicial posets enjoy many of the same properties that the Stanley--Reisner rings and modules of simplicial complexes do. For instance, if $\Delta$ is a simplicial poset of dimension $d-1$, then (1) the Krull dimension of $\KK[\Delta]$ is $d$ and the $\ZZ$-graded Hilbert series of $\KK[\Delta]$ is given by $(1-x)^{-d}\sum_{i=0}^d h_i(\Delta)x^i$ (see \cite[Proposition 3.8]{Stanley-91}); (2) $\KK[\Delta]$ is a finitely generated graded module over $S=\KK[x_v:v \in \Delta, \rk(v)=1]$ and the Betti numbers of $\KK[\Delta]$, $\beta_{i,j}(\KK[\Delta])=\dim_\KK \Tor_i^{{S}}(\KK[\Delta],\KK)_j$, satisfy Hochster's formula that expresses these numbers in terms of simplicial Betti numbers of induced subposets of $\Delta$ (see \cite[Corollary 4.6]{Duval-97}). In fact, it is not hard to see that Duval's proof of this latter fact can be generalized to the case of $\KK[\Delta,\Gamma]$, where $(\Delta,\Gamma)$ is a relative simplicial poset. Consequently, Lemma \ref{2.2} continues to hold for relative simplicial posets. Finally, since face modules of simplicial posets are squarefree modules (this follows from \cite[Lemma 2.3]{Ya} and \cite[Lemma 2.5]{Yanagawa-11}) and since the proof of Theorem \ref{4.1} applies to all squarefree modules, we conclude that Theorem \ref{4.1} holds in the generality of face modules of simplicial posets. 

Therefore, \emph{all} results of Section 4, including Theorem \ref{main3}, continue to hold for relative simplicial posets. Furthermore, the last part of Lemma \ref{mu-properties} combined with Theorem \ref{main3} implies that Theorem \ref{main2} also continues to hold for simplicial posets. In fact, the inequalities of Theorem \ref{main2} are sharp for simplicial posets; this follows from \cite[Theorem 2.5]{Browder-Klee-14} and the Morse inequalities.

{\small
\bibliography{fundamental-biblio}
\bibliographystyle{plain}
}
\end{document}